\documentclass{amsart}
\usepackage{amssymb}
\usepackage{amsmath}
\usepackage{amsfonts}

\newtheorem{theorem}{Theorem}
\theoremstyle{plain}
\newtheorem{acknowledgement}{Acknowledgement}

\newtheorem{corollary}{Corollary}

\newtheorem{example}{Example}

\newtheorem{lemma}{Lemma}

\newtheorem{remark}{Remark}

\begin{document}
\title[]{\textbf{Eigenvalues of the Sturm-Liouville problem with\ a frozen
argument on time scales}}
\email{ }
\author{Zeynep DURNA}
\address{Cumhuriyet University, Faculty of Science, Department of
Mathematics, 58140\\
Sivas, Turkey}
\email{zeynepdurna14@gmail.com}
\author{A. Sinan Ozkan}
\address{Cumhuriyet University, Faculty of Science, Department of
Mathematics, 58140\\
Sivas, Turkey}
\email{sozkan@cumhuriyet.edu.tr }
\subjclass[2000]{45C05, 34N05, 34B24, 34C10}
\keywords{Dynamic equations on time scales or measure chains, Eigenvalue
problems, Sturm-Liouville theory, frozen argument.}

\begin{abstract}
In this study, we consider a boundary value problem generated by the
Sturm-Liouville problem with\ a frozen argument and with non-separated
boundary conditions on a time scale. Firstly, we present some solutions and
characteristic function of the problem on an arbitrary bounded time scale.
Secondly, we prove some properties of eigenvalues and obtain a formulation
for the eigenvalues-number on a finite time scale. Finally, we give an
asymptotic formula for eigenvalues of the problem on another special time
scale: $\mathbb{T}=\left[ \alpha ,\delta _{1}\right] \cup \left[ \delta
_{2},\beta \right] $.
\end{abstract}

\maketitle

\section{\textbf{Introduction}}

A Sturm-Liouville equation with a frozen argument has the form%
\begin{equation*}
-y^{\prime \prime }(t)+q(t)y(a)=\lambda y(t)\text{,}
\end{equation*}%
where $q(t)$ is the potential function, $a$ is the frozen argument and $%
\lambda $ is the complex spectral parameter. The spectral analysis of
boundary value problems generated with this equation is studied in several
publications \cite{Alb}, \cite{Bon}, \cite{But}, \cite{hu}, \cite{Niz} and
references therein. This kind problems are related strongly to non-local
boundary value problems and appear in various applications \cite{Alb2}, \cite%
{Berezin}, \cite{Krall} and \cite{Went}.

A Sturm-Liouville equation with a frozen argument on a time scale $\mathbb{T}
$ can be given as%
\begin{equation}
-y^{\Delta \Delta }(t)+q(t)y(a)=\lambda y^{\sigma }(t),\text{ }t\in \mathbb{T%
}^{\kappa ^{2}}  \label{1}
\end{equation}%
where $y^{\Delta \Delta }$ and $\sigma $ denote the second order $\Delta $%
-derivative of $y$ and forward jump operator on $\mathbb{T}$, respectively, $%
q(t)$ is a real-valued continuous function, $a\in \mathbb{T}^{\kappa }:=%
\mathbb{T}\backslash \left( \rho \left( \sup \mathbb{T}\right) ,\sup \mathbb{%
T}\right] ,$ $y^{\sigma }(t)=y(\sigma (t))$ and $\mathbb{T}^{\kappa
^{2}}=\left( \mathbb{T}^{\kappa }\right) ^{\kappa }.$

Spectral properties the classical Sturm-Liouville problem on time scales
were given in various publications (see e.g. \cite{Adalar}, \cite{Agarwal}, 
\cite{Amster}, \cite{Amster2}, \cite{Davidson}-\cite{Guseinov2}, \cite{Simon}%
-\cite{Kong}, \cite{Oz}-\cite{Sun} and references therein). However, there
is no any publication about the Sturm-Liouville equation with a frozen
argument on a time scale.

In the present paper, we consider a boundary value problem which is
generated by equation (1) and the following boundary conditions%
\begin{eqnarray}
U(y) &:&=a_{11}y\left( \alpha \right) +a_{12}y^{\Delta }\left( \alpha
\right) +a_{21}y\left( \beta \right) +a_{22}y^{\Delta }\left( \beta \right)
\medskip  \label{2} \\
V\left( y\right) &:&=b_{11}y\left( \alpha \right) +b_{12}y^{\Delta }\left(
\alpha \right) +b_{21}y\left( \beta \right) +b_{22}y^{\Delta }\left( \beta
\right) \medskip  \label{3}
\end{eqnarray}%
where $\alpha =\inf \mathbb{T}$, $\beta =\rho (\sup \mathbb{T)}$, $\alpha
\neq \beta $ and $a_{ij}$, $b_{ij}\in 
\mathbb{R}
$ for $i,j=1,2$. We aim to give some properties of some solutions and
eigenvalues of (1)-(3) for two different cases of $\mathbb{T}$

For the basic notation and terminology of time scales theory, we recommend
to see \cite{Atkinson}, \cite{Bohner}, \cite{Bohner2} and \cite{L}.\bigskip

\section{\textbf{Preliminaries}}

Let $S(t,\lambda )$ and $C(t,\lambda )$ be the solutions of (1) under the
initial conditions 
\begin{eqnarray}
S(a,\lambda ) &=&0\text{, }S^{\Delta }(a,\lambda )=1,\medskip \\
C(a,\lambda ) &=&1\text{, }C^{\Delta }(a,\lambda )=0,\medskip
\end{eqnarray}%
respectively. Clearly, $S(t,\lambda )$ and $C(t,\lambda )$ satisfy 
\begin{eqnarray*}
S^{\Delta \Delta }(t,\lambda )+\lambda S^{\sigma }(t,\lambda ) &=&0 \\
C^{\Delta \Delta }(t,\lambda )+\lambda C^{\sigma }(t,\lambda ) &=&q(t),
\end{eqnarray*}%
respectively and so these functions and their $\Delta $-derivatives are
entire on $\lambda $ for each fixed $t$ (see \cite{Oz}).

\begin{lemma}
Let $\varphi (t,\lambda )$ be the solution of (1) under the initial
conditions $\varphi (a,\lambda )=\delta _{1},$ $\varphi ^{\Delta }(a,\lambda
)=\delta _{2}$ for given numbers $\delta _{1},\delta _{2}.$ Then $\varphi
(t,\lambda )=\delta _{1}C(t,\lambda )+\delta _{2}S(t,\lambda )$ is valid on $%
\mathbb{T}$.
\end{lemma}

\begin{proof}
It is clear that the function $y(t,\lambda )=\delta _{1}C(t,\lambda )+\delta
_{2}S(t,\lambda )$ is the solution of the initial value problem 
\begin{eqnarray*}
y^{\Delta \Delta }(t)+\lambda y^{\sigma }(t) &=&q(t)\delta _{1} \\
y(a,\lambda ) &=&\delta _{1} \\
y^{\Delta }(a,\lambda ) &=&\delta _{2}.
\end{eqnarray*}%
We obtain by taking into account uniqueness of the solution of an initial
value problem that $y(t,\lambda )=\varphi (t,\lambda ).$
\end{proof}

Consider the function%
\begin{equation}
\Delta (\lambda ):\det \left( 
\begin{array}{cc}
U(C) & V(C) \\ 
U(S) & V(S)%
\end{array}%
\right) .
\end{equation}%
It is obvious $\Delta (\lambda )$ is also entire.

\begin{theorem}
The zeros of the function $\Delta (\lambda )$ coincide with the eigenvalues
of the problem (1)-(3).
\end{theorem}

\begin{proof}
Let $\lambda _{0}$ be an eigenvalue and $y(t,\lambda _{0})=\delta
_{1}C(t,\lambda _{0})+\delta _{2}S(t,\lambda _{0})$ is the corresponding
eigenfunction, then $y(t,\lambda _{0})$ satisfies (2) and (3). Therefore, 
\begin{eqnarray*}
&&\left. \delta _{1}U(C(t,\lambda _{0}))+\delta _{2}U(S(t,\lambda
_{0}))=0,\right. \medskip \\
&&\left. \delta _{1}V(C(t,\lambda _{0}))+\delta _{2}V(S(t,\lambda
_{0}))=0.\right. \medskip
\end{eqnarray*}%
It is obvious that $y(t,\lambda _{0})\neq 0$ iff the
coefficients-determinant of the above system vanishes, i.e., $\Delta
(\lambda _{0})=0.$\bigskip
\end{proof}

Since $\Delta (\lambda )$ is an entire function, eigenvalues of the problem
(1)-(3) are discrete.\bigskip

\section{\textbf{Eigenvalues of (1)-(3) on a finite time scale}}

Let $\mathbb{T}$ be a finite time scale such that there are $m$ (or $r$)
many elements which are larger (or smaller) than $a$ in $\mathbb{T}$. Assume 
$m\geq 1$, $r\geq 0$ and $r+m\geq 2.$ It is clear that the number of
elements of $\mathbb{T}$ is $n=m+r+1.$ We can write $\mathbb{T}$ as follows 
\begin{equation*}
\mathbb{T=}\left\{ \rho ^{r}\left( a\right) ,\rho ^{r-1}\left( a\right)
,...,\rho ^{2}\left( a\right) ,\rho \left( a\right) ,a,\sigma (a),\sigma
^{2}(a),...,\sigma ^{m-1}(a),\sigma ^{m}(a)\right\} ,
\end{equation*}%
where $\sigma ^{j}=\sigma ^{j-1}\circ \sigma $, $\rho ^{j}=\rho ^{j-1}\circ
\rho $ for $j\geq 2$, $\rho ^{r}\left( a\right) =\alpha $ and $\sigma
^{m-1}(\alpha )=\beta .$ $\bigskip $

\begin{lemma}
i) If $r\geq 3$ and $m\geq 2$, the following equalities hold for all $%
\lambda \bigskip $\newline
$S(\alpha ,\lambda )=\left( -1\right) ^{r}\mu ^{\rho }\left( a\right) \left[
\mu ^{\rho ^{2}}\left( a\right) \mu ^{\rho ^{3}}\left( a\right) ...\mu
^{\rho ^{r}}\left( a\right) \right] ^{2}\lambda ^{r-1}+O\left( \lambda
^{r-2}\right) \bigskip $\newline
$S^{\sigma }(\alpha ,\lambda )=\left( -1\right) ^{r-1}\mu ^{\rho }\left(
a\right) \left[ \mu ^{\rho ^{2}}\left( a\right) \mu ^{\rho ^{3}}\left(
a\right) ...\mu ^{\rho ^{r-1}}\left( a\right) \right] ^{2}\lambda
^{r-2}+O\left( \lambda ^{r-3}\right) \bigskip $\newline
$S\left( \beta ,\lambda \right) =S^{\sigma ^{m-1}}\left( a,\lambda \right)
=\left( -1\right) ^{m}\left[ \mu \left( a\right) \mu ^{\sigma }\left(
a\right) ...\mu ^{\sigma ^{m-3}}\left( a\right) \right] ^{2}\lambda
^{m-2}\mu ^{\sigma ^{m-2}}\left( a\right) +O\left( \lambda ^{m-3}\right)
\bigskip $\newline
$S^{\sigma }\left( \beta ,\lambda \right) =S^{\sigma ^{m}}\left( a,\lambda
\right) =\left( -1\right) ^{m+1}\left[ \mu \left( a\right) \mu ^{\sigma
}\left( a\right) ...\mu ^{\sigma ^{m-2}}\left( a\right) \right] ^{2}\lambda
^{m-1}\mu ^{\sigma ^{m-1}}\left( a\right) +O\left( \lambda ^{m-2}\right)
\bigskip $\newline
$C\left( \alpha ,\lambda \right) =\left( -1\right) ^{r}\left[ \mu ^{\rho
}\left( a\right) \mu ^{\rho ^{2}}\left( a\right) ...\mu ^{\rho ^{r}}\left(
a\right) \right] ^{2}\lambda ^{r}+O\left( \lambda ^{r-1}\right) \bigskip $%
\newline
$C^{\sigma }\left( \alpha ,\lambda \right) =\left( -1\right) ^{r-1}\left[
\mu ^{\rho }\left( a\right) \mu ^{\rho ^{2}}\left( a\right) ...\mu ^{\rho
^{r-1}}\left( a\right) \right] ^{2}\lambda ^{r-1}+O\left( \lambda
^{r-2}\right) \bigskip $\newline
$C(\beta ,\lambda )=C^{\sigma ^{m-1}}\left( a,\lambda \right) =\left(
-1\right) ^{m}\mu \left( a\right) \left[ \mu ^{\sigma }\left( a\right) \mu
^{\sigma ^{2}}\left( a\right) ...\mu ^{\sigma ^{m-3}}\left( a\right) \right]
^{2}\mu ^{\sigma ^{m-2}}\left( a\right) \lambda ^{m-2}+O\left( \lambda
^{m-3}\right) \bigskip $\newline
$C^{\sigma }(\beta ,\lambda )=C^{\sigma ^{m}}\left( a,\lambda \right)
=\left( -1\right) ^{m+1}\mu \left( a\right) \left[ \mu ^{\sigma }\left(
a\right) \mu ^{\sigma ^{2}}\left( a\right) ...\mu ^{\sigma ^{m-2}}\left(
a\right) \right] ^{2}\mu ^{\sigma ^{m-1}}\left( a\right) \lambda
^{m-1}+O\left( \lambda ^{m-2}\right) ,\bigskip $\newline
where $O(\lambda ^{l})$ denotes polynomials whose degrees are $l$. $\bigskip 
$\newline
ii) If $r\in \{0,1,2\}$ or $m\in \{0,1\}$, degrees of all above functions
are vanish.
\end{lemma}

\begin{proof}
It is clear from $f^{\sigma }(t)=f(t)+\mu (t)f^{\Delta }(t)$ that\ $%
S^{\sigma }(a,\lambda )=\mu (a)$ and $C^{\sigma }(a,\lambda )=1.$ On the
other hand, since $S(t,\lambda )$ and $C(t,\lambda )$ satisfy (1) then the
following equalities hold for each $t\in \mathbb{T}^{\kappa }$ and for all $%
\lambda .$%
\begin{eqnarray}
S^{\sigma ^{2}}\left( t,\lambda \right) &=&\left( 1+\frac{\mu \left(
t\right) }{\mu ^{\sigma }(t)}-\lambda \mu \left( t\right) \mu ^{\sigma
}\left( t\right) -\lambda \mu ^{2}\left( t\right) \mu ^{\sigma }\left(
t\right) \right) S^{\sigma }\left( t,\lambda \right)  \label{7} \\
&&-\mu ^{\sigma }\left( t\right) S\left( t,\lambda \right) \bigskip  \notag
\\
C^{\sigma ^{2}}(t,\lambda ) &=&\left( -\mu \left( t\right) \mu ^{\sigma
}\left( t\right) \lambda +\frac{\mu \left( t\right) +\mu ^{\sigma }\left(
t\right) }{\mu \left( t\right) }\right) C^{\sigma }(t,\lambda )  \label{8} \\
&&+\mu \left( t\right) \mu ^{\sigma }\left( t\right) q(t)  \notag
\end{eqnarray}

It can be calculated from (7) and (8) that%
\begin{eqnarray}
S^{\sigma ^{j}}(a,\lambda ) &=&\left( -1\right) ^{j+1}\left( \mu (a)\mu
^{\sigma }(a)...\mu ^{\sigma ^{j-2}}(a)\right) ^{2}\mu ^{\sigma
^{j-1}}(a)\lambda ^{j-1}  \label{9} \\
&&+O\left( \lambda ^{j-2}\right) \bigskip  \notag \\
S^{\rho ^{j}}(a,\lambda ) &=&\left( -1\right) ^{j}\mu ^{\rho }(a)\left( \mu
^{\rho ^{2}}(a)\mu ^{\rho ^{3}}(a)...\mu ^{\rho ^{j}}(a)\right) ^{2}\lambda
^{j-1} \\
&&+O\left( \lambda ^{j-2}\right) \bigskip  \notag \\
C^{\sigma ^{k}}(a,\lambda ) &=&\left( -1\right) ^{k+1}\mu \left( a\right)
\left( \mu ^{\sigma }\left( a\right) \mu ^{\sigma ^{2}}\left( a\right)
...\mu ^{\sigma ^{k-2}}\left( a\right) \right) ^{2}\mu ^{\sigma
^{k-1}}\left( a\right) \lambda ^{k-1} \\
&&+O\left( \lambda ^{k-2}\right) \bigskip  \notag \\
C^{\rho ^{k}}\left( a,\lambda \right) &=&\left( -1\right) ^{k}\left( \mu
^{\rho }\left( a\right) \mu ^{\rho ^{2}}\left( a\right) ...\mu ^{\rho
^{k}}\left( a\right) \right) ^{2}\lambda ^{k}  \label{12} \\
&&+O\left( \lambda ^{k-1}\right)  \notag
\end{eqnarray}%
for $j=2,3,...m$ and $k=2,3,...,r$. Using (9)-(12) and taking into account $%
\alpha =\rho ^{r}\left( a\right) $ and $\beta =\sigma ^{m-1}(\alpha )$ we
have our desired relations.\bigskip
\end{proof}

\begin{corollary}
$\deg C(\alpha ,\lambda )S^{\sigma }\left( \beta ,\lambda \right) =\left\{ 
\begin{array}{c}
r+m-1,\text{ }r>0\text{ and }m>1 \\ 
1,\text{ \ \ \ \ \ \ \ \ the other cases}%
\end{array}%
\right. ,$ \bigskip \newline
\end{corollary}

\begin{lemma}
The following equlaties hold for all $\lambda \in 
\mathbb{C}
.$%
\begin{eqnarray*}
S^{\sigma }(\alpha ,\lambda )C\left( \alpha ,\lambda \right) -S(\alpha
,\lambda )C^{\sigma }\left( \alpha ,\lambda \right) &=&A\lambda ^{\delta
}+O\left( \lambda ^{\delta -1}\right) \bigskip \\
S^{\sigma }\left( \beta ,\lambda \right) C\left( \beta ,\lambda \right)
-S(\beta ,\lambda )C^{\sigma }\left( \beta ,\lambda \right) &=&B\lambda
^{\gamma }+O\left( \lambda ^{\gamma -1}\right) \bigskip
\end{eqnarray*}%
where $A=\left( -1\right) ^{r}\mu \left( \alpha \right) \mu ^{\rho }\left(
a\right) \left[ \mu ^{\rho ^{2}}\left( a\right) ...\mu ^{\rho ^{r-1}}\left(
a\right) \right] ^{2}\mu ^{\rho ^{r}}\left( a\right) q\left( \alpha \right) $%
, \bigskip \newline
$B=\left( -1\right) ^{m-1}\mu \left( \beta \right) \left[ \mu \left(
a\right) \mu ^{\sigma }\left( a\right) ...\mu ^{\sigma ^{m-2}}\left(
a\right) \right] ^{2}q\left( \rho \left( \beta \right) \right) $, \bigskip 
\newline
$\delta =\left\{ 
\begin{array}{c}
r-2,\text{ \ \ }r\geq 3 \\ 
0,\text{ \ \ }r<3%
\end{array}%
\right. $ and $\gamma =\left\{ 
\begin{array}{c}
m-2,\text{ \ \ }m\geq 3 \\ 
0,\text{ \ \ }m<3.%
\end{array}%
\right. $
\end{lemma}

\begin{proof}
Consider the function 
\begin{equation}
\varphi \left( t,\lambda \right) :=\frac{1}{\mu \left( t\right) }\left[
S^{\sigma }(t,\lambda )C\left( t,\lambda \right) -S(t,\lambda )C^{\sigma
}\left( t,\lambda \right) \right]
\end{equation}%
It is clear that%
\begin{equation*}
\varphi \left( t,\lambda \right) :=\left[ S^{\Delta }(t,\lambda )C\left(
t,\lambda \right) -S(t,\lambda )C^{\Delta }\left( t,\lambda \right) \right]
=W\left[ C\left( t,\lambda \right) ,S\left( t,\lambda \right) \right]
\end{equation*}%
and it is the solution of initial value problem%
\begin{eqnarray*}
\varphi ^{\Delta }\left( t\right) &=&-q\left( t\right) S^{\sigma }\left(
t,\lambda \right) \bigskip \\
\varphi \left( a\right) &=&1
\end{eqnarray*}%
Therefore, we can obtain the following relations%
\begin{eqnarray}
\varphi ^{\sigma }\left( t,\lambda \right) &=&\varphi \left( t,\lambda
\right) -\mu \left( t\right) q\left( t\right) S^{\sigma }\left( t,\lambda
\right) ,\bigskip  \label{14} \\
\varphi ^{\rho }\left( t,\lambda \right) &=&\varphi \left( t,\lambda \right)
+\mu ^{\rho }\left( t\right) q\left( \rho \left( t\right) \right) S\left(
t,\lambda \right) .  \label{15}
\end{eqnarray}%
By using (9), (10), (14) and (15), the proof is completed.\bigskip
\end{proof}

\begin{corollary}
i) $\deg \left( S^{\sigma }\left( \alpha ,\lambda \right) C(\alpha ,\lambda
)-S\left( \alpha ,\lambda \right) C^{\sigma }\left( \alpha ,\lambda \right)
\right) <\deg C(\alpha ,\lambda )S^{\sigma }\left( \beta ,\lambda \right) ,$
\bigskip\ \newline
ii) $\deg \left( S^{\sigma }\left( \beta ,\lambda \right) C\left( \beta
,\lambda \right) -S\left( \beta ,\lambda \right) C^{\sigma }\left( \beta
,\lambda \right) \right) <\deg C(\alpha ,\lambda )S^{\sigma }\left( \beta
,\lambda \right) .$\bigskip
\end{corollary}

The next theorem gives the number of eigenvalues of the problem (1)-(3) on $%
\mathbb{T}$. Recall $n=m+r+1$ denotes the number of elements of $\mathbb{T}$
and put $A=\left( 
\begin{array}{cc}
a_{11}\mu \left( \alpha \right) -a_{12} & b_{11}\mu \left( \alpha \right)
-b_{12} \\ 
a_{22} & b_{22}%
\end{array}%
\right) .\medskip $

\begin{theorem}
If $\det A\neq 0$, the problem (1)-(3) has exactly $n-2$ many eigenvalues
with multiplications, otherwise the eigenvalues-number of (1)-(3) is least
than $n-2$.$\bigskip $
\end{theorem}

\begin{proof}
Since $\mathbb{T}$ is finite, $\Delta (\lambda )$ is a polinomial and its
degree gives the number eigenvalues of the problem. It can be calculated
from (6)-(14) that\bigskip 
\begin{eqnarray*}
\Delta (\lambda ) &=&\frac{1}{\mu \left( \alpha \right) \mu \left( \beta
\right) }\det \left( 
\begin{array}{cc}
a_{11}\mu \left( \alpha \right) -a_{12} & b_{11}\mu \left( \alpha \right)
-b_{12} \\ 
a_{22} & b_{22}%
\end{array}%
\right) C(\alpha ,\lambda )S^{\sigma }\left( \beta ,\lambda \right) +\bigskip
\\
&&+\frac{1}{\mu \left( \alpha \right) }\det \left( 
\begin{array}{cc}
a_{11} & a_{12} \\ 
b_{11} & b_{12}%
\end{array}%
\right) \left( S^{\sigma }\left( \alpha ,\lambda \right) C(\alpha ,\lambda
)-S\left( \alpha ,\lambda \right) C^{\sigma }\left( \alpha ,\lambda \right)
\right) +\bigskip \\
&&+\frac{1}{\mu \left( \beta \right) }\det \left( 
\begin{array}{cc}
a_{21} & a_{22} \\ 
b_{21} & b_{22}%
\end{array}%
\right) \left( S^{\sigma }\left( \beta ,\lambda \right) C\left( \beta
,\lambda \right) -S\left( \beta ,\lambda \right) C^{\sigma }\left( \beta
,\lambda \right) \right) +O(\lambda ^{n+m-2}).\bigskip
\end{eqnarray*}%
According to Corollary 1 and Corollary 2, if $\det A\neq 0$, $\deg \Delta
(\lambda )=\deg C(\alpha ,\lambda )S^{\sigma }\left( \beta ,\lambda \right)
=m+r-1=n-2.\bigskip $
\end{proof}

\begin{corollary}
i) The eigenvalues-number of (1)-(3) depends only on the elements-number of $%
\mathbb{T}$ and the coefficients of the boundary conditions (2) and (3). On
the other hand, it does not depend on $q(t)$ and $a$ (neither value nor
location of $a$ on $\mathbb{T}$).\newline
ii) If $\det A\neq 0$, the eigenvalues-number of (1)-(3) and the
elements-number of $\mathbb{T}$ determine uniquely each other.
\end{corollary}

\begin{example}[Separated boundary conditions]
Let us consider the time scale $\mathbb{T=}\left\{ 0,1,2,...,n-1\right\} $
and the following boundary value problem which appears some applications.%
\begin{eqnarray}
&&\left. -y^{\Delta \Delta }\medskip \left( t\right) +q\left( t\right)
y\left( a\right) =\lambda y\left( t+1\right) \text{, }t\in \mathbb{T}%
^{\kappa ^{2}}=\left\{ 0,1,...,n-3\right\} \right.  \label{16} \\
&&\left. y^{\Delta }\medskip \left( 0\right) +hy\left( 0\right) =0\right. \\
&&\left. y^{\Delta }\medskip \left( n-2\right) +Hy\left( n-2\right)
=0,\right.
\end{eqnarray}%
where $0\leq a\leq n-2$ and $h$, $H\in 
\mathbb{R}
.$ According to Theorem 2, if $h\neq 1$, the eigenvalues-number of this
problem is exactly $n-2$, otherwise less than $n-2.$
\end{example}

Now, we want to give a theorem which includes some informations about the
eigenvalues of (16)-(18).

\begin{theorem}
Let $Q=Q_{1}+Q_{2},\medskip $ $\ \ $\newline
$Q_{1}=\left( 
\begin{array}{cccccccccc}
2-\frac{1}{1-h} & -1 & 0 & 0 & 0 & ... & 0 & 0 & 0 & 0 \\ 
-1 & 2 & -1 & 0 & 0 & ... & 0 & 0 & 0 & 0 \\ 
0 & -1 & 2 & -1 & 0 & ... & 0 & 0 & 0 & 0 \\ 
. & . & . & . & . & . & . & . & . & . \\ 
. & . & . & . & . & . & . & . & . & . \\ 
. & . & . & . & . & . & . & . & . & . \\ 
0 & 0 & 0 & 0 & 0 & -1 & 2 & -1 & 0 & 0 \\ 
0 & 0 & 0 & 0 & 0 & 0 & -1 & 2 & -1 & 0 \\ 
0 & 0 & 0 & 0 & 0 & 0 & 0 & -1 & 2 & -1 \\ 
0 & 0 & 0 & 0 & 0 & 0 & 0 & 0 & -1 & 1-H%
\end{array}%
\right) _{(n-2)\times (n-2)\medskip }$\newline
and $Q_{2}=\left( q_{ij}\right) _{(n-2)\times (n-2)},$ where $q_{ij}=\left\{ 
\begin{array}{c}
q(i-1),\text{ \ \ }j=a \\ 
0,\text{ \ \ }j\neq a%
\end{array}%
\right. $ and $h\in 
\mathbb{R}
-\{1\}$.$\medskip $ \newline
The problem (16)-(18) and the matrix $Q$ have the same eigenvalues.
\end{theorem}

\begin{proof}
Let $y(t)$ be an eigenfunction of (16)-(18). Since $y^{\Delta }\left(
t\right) =y(t+1)-y(t)$ for $t\in \mathbb{T}^{\kappa }$ and $y^{\Delta \Delta
}\left( t\right) =y(t+2)-2y(t+1)+y(t)$ for $t\in \mathbb{T}^{\kappa ^{2}}$,
we can write (16)-(18) as follows 
\begin{eqnarray}
y(t+2)-2y(t+1)+y(t) &=&q(t)y(a)-\lambda y\left( t+1\right) ,\text{ }t\in 
\mathbb{T}^{\kappa ^{2}},\medskip  \label{19} \\
y(1) &=&(h-1)y(0),\medskip \\
y(n) &=&(1-H)y(n-1).\medskip
\end{eqnarray}%
We obtain from (19)-(21) a linear system whose coefficients-matrix is $%
\lambda I-Q$. Therefore it is concluded that the boundary value problem
(16)-(18) and the matrix $Q$ have the same eigenvalues.
\end{proof}

\begin{remark}
The result in the Theorem 3 can be generalized easily to (1)-(3) on the
general discrete time scale.
\end{remark}

\begin{remark}
As is known, all eigenvalues of the classical Sturm-Liouville problem with
separated boundary conditions on time scales are real and algebraicly simple 
\cite{Agarwal}. However, the Sturm-Liouville problem with the frozen
argument may have non-real or non-simple eigenvalues even if it is equipped
with separated boundary conditions.
\end{remark}

We end this section with two examples. The problem in the first example has
non-real or non-simple eigenvalues, unlike, all eigenvalues of the latter
problem are real and simple.

\begin{example}
Consider the following problem on $\mathbb{T}=\{0,1,2,3,4,5\}$.%
\begin{equation*}
L_{1}:\left\{ 
\begin{array}{c}
-y^{\Delta \Delta }(t)+q_{1}(t)y(3)=\lambda y^{\sigma }(t),\text{ }t\in
\{0,1,2,3\}\medskip \\ 
y^{\Delta }(0)+\frac{1}{2}y(0)=0\medskip \\ 
y^{\Delta }(4)+y(4)=0,%
\end{array}%
\right.
\end{equation*}%
where $q_{1}(t)=\left\{ 
\begin{array}{cc}
-3 & t=0 \\ 
10 & t=1 \\ 
-5 & t=2 \\ 
1 & t=3%
\end{array}%
\right. $. According to Theorem 3, eigenvalues of $L_{1}$ coincide with
eigenvalues of the matrix $Q_{1}=\left( 
\begin{array}{cccc}
0 & -1 & -3 & 0 \\ 
-1 & 2 & 9 & 0 \\ 
0 & -1 & -3 & -1 \\ 
0 & 0 & 0 & 0%
\end{array}%
\right) $ and they are $\lambda _{1}=\lambda _{2}=0,$ $\lambda _{3}=-\frac{1%
}{2}+\frac{1}{2}i\sqrt{7},$ $\lambda _{4}=-\frac{1}{2}-\frac{1}{2}i\sqrt{7}.$
\end{example}

\begin{example}
Consider the following problem on $\mathbb{T}=\{0,1,2,3,4,5\}$.%
\begin{equation*}
L_{2}:\left\{ 
\begin{array}{c}
-y^{\Delta \Delta }(t)+q_{2}(t)y(4)=\lambda y^{\sigma }(t),\text{ }t\in
\{0,1,2,3\}\medskip \\ 
y^{\Delta }(0)=0\medskip \\ 
y^{\Delta }(4)=0,%
\end{array}%
\right.
\end{equation*}%
where $q_{2}(t)=t$. \newline
Clearly, eigenvalues of $L_{2}$ coincide with eigenvalues of $Q_{2}=\left( 
\begin{array}{cccc}
1 & -1 & 0 & 0 \\ 
-1 & 2 & -1 & 1 \\ 
0 & -1 & 2 & 1 \\ 
0 & 0 & -1 & 4%
\end{array}%
\right) $ and they are $\lambda _{1}=2+\sqrt{3},$ $\lambda _{2}=2-\sqrt{3},$ 
$\lambda _{3}=3,$ $\lambda _{4}=2$.\bigskip
\end{example}

\section{\textbf{Eigenvalues of (1)-(3) on the time scale }$\mathbb{T}=\left[
\protect\alpha ,\protect\delta _{1}\right] \cup \left[ \protect\delta _{2},%
\protect\beta \right] $}

In this section, we investigate eigenvalues of the problem (1)-(3) on
another special time scale: $\mathbb{T}=\left[ \alpha ,\delta _{1}\right]
\cup \left[ \delta _{2},\beta \right] $, where $\alpha <a<\delta _{1}<\delta
_{2}<\beta $. We assume that $a\in \left( \alpha ,\delta _{1}\right) .$ The
similar results can be obtained in the case when $a\in \left( \delta
_{2},\beta \right) $.

The following relations are valid on $\left[ \alpha ,\delta _{1}\right] $
(see \cite{Bon}). 
\begin{eqnarray*}
S(t,\lambda ) &=&\dfrac{\sin \sqrt{\lambda }\left( t-a\right) }{\sqrt{%
\lambda }} \\
C(t,\lambda ) &=&\cos \sqrt{\lambda }\left( t-a\right) +\int\limits_{a}^{t}%
\dfrac{\sin \sqrt{\lambda }\left( t-\xi \right) }{\sqrt{\lambda }}q(\xi )d\xi
\end{eqnarray*}

The following asymptotic relations for the solutions $S(t,\lambda )$ and $%
C(t,\lambda )$ can be proved by using a method similar to one in \cite{Oz2}.%
\begin{equation}
S(t,\lambda )=\left\{ 
\begin{array}{c}
\dfrac{\sin \sqrt{\lambda }\left( t-a\right) }{\sqrt{\lambda }},\text{ \ \ \ 
}t\in \mathbb{[}\alpha ,\delta _{1}],\bigskip \\ 
\delta ^{2}\sqrt{\lambda }\cos \sqrt{\lambda }\left( \delta _{1}-a\right)
\sin \sqrt{\lambda }(\delta _{2}-t)+O\left( \exp \left\vert \tau \right\vert
(\delta +a-t)\right) ,\text{ }t\in \lbrack \delta _{2},\beta ],%
\end{array}%
\right.  \label{22}
\end{equation}%
\begin{equation}
S^{\Delta }(t,\lambda )=\left\{ 
\begin{array}{c}
\cos \sqrt{\lambda }\left( t-a\right) ,\text{ \ \ \ }t\in \mathbb{[}\alpha
,\delta _{1}),\bigskip \\ 
-\delta ^{2}\lambda \cos \sqrt{\lambda }\left( \delta _{1}-a\right) \cos 
\sqrt{\lambda }(\delta _{2}-t)+O\left( \sqrt{\lambda }\exp \left\vert \tau
\right\vert \left( \delta +a-t\right) \right) ,\text{ }t\in \lbrack \delta
_{2},\beta ],%
\end{array}%
\right.
\end{equation}%
\begin{equation}
C(t,\lambda )=\left\{ 
\begin{array}{c}
\cos \sqrt{\lambda }\left( t-a\right) +O\left( \dfrac{1}{\sqrt{\lambda }}%
\exp \left\vert \tau \right\vert \left\vert t-a\right\vert \right) ,\text{ \
\ \ }t\in \mathbb{[}\alpha ,\delta _{1}],\bigskip \\ 
-\delta ^{2}\lambda \sin \sqrt{\lambda }\left( \delta _{1}-a\right) \sin 
\sqrt{\lambda }(\delta _{2}-t)+O\left( \sqrt{\lambda }\exp \left\vert \tau
\right\vert \left( \delta +a-t\right) \right) ,\text{ }t\in \lbrack \delta
_{2},\beta ],%
\end{array}%
\right.  \label{24}
\end{equation}%
\begin{equation}
C^{\Delta }(t,\lambda )=\left\{ 
\begin{array}{c}
-\sqrt{\lambda }\sin \sqrt{\lambda }\left( t-a\right) +O\left( \exp
\left\vert \tau \right\vert \left\vert t-a\right\vert \right) ,\text{ \ \ \ }%
t\in \mathbb{[}\alpha ,\delta _{1}),\bigskip \\ 
\delta ^{2}\lambda ^{3/2}\sin \sqrt{\lambda }\left( \delta _{1}-a\right)
\cos \sqrt{\lambda }(\delta _{2}-t)+O\left( \lambda \exp \left\vert \tau
\right\vert \left( \delta +a-t\right) \right) ,\text{ }t\in \lbrack \delta
_{2},\beta ],%
\end{array}%
\right.
\end{equation}%
where $\delta =\delta _{2}-\delta _{1},$ $\tau =$Im$\sqrt{\lambda }$ and $O$
denotes Landau's symbol. \bigskip

\begin{lemma}
The following equlaties hold for all $\lambda \in 
\mathbb{C}
$ and $t\in \mathbb{T}.$%
\begin{equation*}
C^{\Delta }(t,\lambda )S\left( t,\lambda \right) -C(t,\lambda )S^{\Delta
}\left( t,\lambda \right) =O\left( \exp \left\vert \tau \right\vert \left(
\beta -\alpha -\delta \right) \right)
\end{equation*}
\end{lemma}

\begin{proof}
It is clear the function 
\begin{equation*}
\varphi \left( t,\lambda \right) :=C^{\Delta }(t,\lambda )S\left( t,\lambda
\right) -C(t,\lambda )S^{\Delta }\left( t,\lambda \right)
\end{equation*}%
satisfies initial value problem%
\begin{eqnarray*}
\varphi ^{\Delta }\left( t\right) &=&q\left( t\right) S^{\sigma }\left(
t,\lambda \right) ,\text{ }t\in \left[ \alpha ,\delta _{1}\right] \\
\varphi \left( a\right) &=&1
\end{eqnarray*}%
and%
\begin{eqnarray*}
\varphi ^{\Delta }\left( t\right) &=&q\left( t\right) S^{\sigma }\left(
t,\lambda \right) ,\text{ }t\in \left[ \delta _{2},\beta \right] \\
\varphi \left( \delta _{2}\right) &=&\varphi \left( \delta _{1}\right)
+\delta q(\delta _{1})S\left( \delta _{2},\lambda \right) .
\end{eqnarray*}
Hence, we get proof by using (22).\bigskip
\end{proof}

\begin{theorem}
i) The problem (1)-(3) on $\mathbb{T}=\left[ \alpha ,\delta _{1}\right] \cup %
\left[ \delta _{2},\beta \right] $ has countable many eigenvalues such as $%
\left\{ \lambda _{n}\right\} _{n\geq 0}$. \newline
ii) The numbers $\left\{ \lambda _{n}\right\} _{n\geq 0}$ are real for
sufficiently large $n$. \newline
iii) If $a_{22}b_{12}-a_{12}b_{22}\neq 0$ and $\beta -\delta _{2}=\delta
_{1}-\alpha $, the following asymptotic formula holds for $n\rightarrow
\infty $.%
\begin{equation}
\sqrt{\lambda _{n}}=\frac{(n-1)\pi }{2\left( \beta -\delta _{2}\right) }%
+O\left( \dfrac{1}{n}\right)
\end{equation}
\end{theorem}

\begin{proof}
The proof of (i) is obvious, since $\Delta (\lambda )$ is entire on $\lambda 
$.

By calculating directly, we get%
\begin{eqnarray*}
\Delta (\lambda ) &=&\det \left( 
\begin{array}{cc}
U(C) & V(C) \\ 
U(S) & V(S)%
\end{array}%
\right) \\
&=&(a_{22}b_{12}-a_{12}b_{22})\left[ C^{\Delta }(\beta ,\lambda )S^{\Delta
}\left( \alpha ,\lambda \right) -C^{\Delta }(\alpha ,\lambda )S^{\Delta
}\left( \beta ,\lambda \right) \right] +\bigskip \\
&&+(a_{22}b_{21}-a_{21}b_{22})\left[ C^{\Delta }(\beta ,\lambda )S\left(
\beta ,\lambda \right) -C(\beta ,\lambda )S^{\Delta }\left( \beta ,\lambda
\right) \right] +\bigskip \\
&&+(a_{12}b_{11}-a_{11}b_{12})\left[ C^{\Delta }(\alpha ,\lambda )S\left(
\alpha ,\lambda \right) -C(\alpha ,\lambda )S^{\Delta }\left( \alpha
,\lambda \right) \right] \\
&&+O\left( \lambda \exp \left\vert \tau \right\vert \left( \beta -\alpha
-\delta \right) \right) .
\end{eqnarray*}%
It follows from (22)-(25) and Lemma 4 that 
\begin{eqnarray*}
\Delta (\lambda ) &=&(a_{22}b_{12}-a_{12}b_{22})\delta ^{2}\lambda
^{3/2}\sin \sqrt{\lambda }(\delta _{1}-\alpha )\cos \sqrt{\lambda }(\beta
-\delta _{2}) \\
&&+O\left( \lambda \exp \left\vert \tau \right\vert \left( \beta -\alpha
-\delta \right) \right)
\end{eqnarray*}%
is valid for $\left\vert \lambda \right\vert \rightarrow \infty .$ Since $%
a_{22}b_{12}-a_{12}b_{22}\neq 0$ and $\beta -\delta _{2}=\delta _{1}-\alpha $%
, the numbers $\left\{ \lambda _{n}\right\} _{n\geq 0}$ are roots of 
\begin{equation}
\lambda ^{2}\frac{\sin 2\sqrt{\lambda }(\beta -\delta _{2})}{\sqrt{\lambda }}%
+O\left( \lambda \exp 2\left\vert \tau \right\vert (\beta -\delta
_{2})\right) =0.
\end{equation}%
Now, we consider the region%
\begin{equation*}
G_{n}:=\{\lambda \in 
\mathbb{C}
:\lambda =\rho ^{2},\left\vert \rho \right\vert <\frac{n\pi }{2\left( \beta
-\delta _{2}\right) }+\varepsilon \}
\end{equation*}%
where $\varepsilon $ is sufficiently small number. There exist some positive
constants $C_{\varepsilon }$ such that, $\left\vert \lambda ^{2}\frac{\sin 2%
\sqrt{\lambda }(\beta -\delta _{2})}{\sqrt{\lambda }}\right\vert \geq
C_{\varepsilon }\left\vert \lambda \right\vert ^{3/2}\exp 2\left\vert \tau
\right\vert (\beta -\delta _{2})$ for sufficiently large $\lambda \in
\partial G_{n}.$ Thus, we establish the proof of (ii). On the other hand by
using Rouche's theorem to (27) on $G_{n}$, we can show clearly that (26)
holds for sufficiently large $n$.
\end{proof}

\begin{remark}
Since $\mu \left( \alpha \right) =0$ in the considered time scale, the term $%
a_{22}b_{12}-a_{12}b_{22}$ is not another than $detA$ in section 3.
\end{remark}

\begin{acknowledgement}
This work does not have any conflicts of interest.
\end{acknowledgement}


\begin{thebibliography}{99}
\bibitem{Adalar} Adalar, \.{I}., Ozkan, A.S.: An interior inverse
Sturm--Liouville problem on a time scale, Analysis and Mathematical Physics,
10, 1-9, (2020)

\bibitem{Agarwal} Agarwal, R.P., Bohner, M., Wong, P.J.Y.: Sturm-Liouville
eigenvalue problems on time scales. Appl. Math. Comput. 99, 153--166 (1999)

\bibitem{Alb} Albeverio S., Hryniv, R.O., Nizhink, L.P.: Inverse Spectral
Problems for non-local Strum Liouville operators, 2007-523-535, (1975)

\bibitem{Alb2} Albeverio S., Nizhnik, L.: Schr\"{o}dinger operators with
nonlocal point interactions J. Math. Anal. Appl., 332(2), 884-895, (2007)

\bibitem{Amster} Amster, P., De N\'{a}poli, P., Pinasco, J.P.: Eigenvalue
distribution of second-order dynamic equations on time scales considered as
fractals. J. Math. Anal. Appl. 343, 573--584 (2008)

\bibitem{Amster2} Amster, P., De N\'{a}poli, P., Pinasco, J.P.: Detailed
asymptotic of eigenvalues on time scales, J. Differ. Equ. Appl. 15 pp.
225--231 (2009)

\bibitem{Atkinson} Atkinson, F.: Discrete and Continuous Boundary Problems.
Academic Press, New York (1964)

\bibitem{Berezin} Berezin, F.A. and Faddeev, L.D.: Remarks on Schr%
\"{}%
odinger equation Sov. Math.---Dokl. 137, 1011--4, (1961)

\bibitem{Bohner} Bohner, M., Peterson, A.: Dynamic Equations on Time Scales.
Birkh\"{a}user, Boston, MA (2001)

\bibitem{Bohner2} Bohner, M., Peterson, A.: Advances in Dynamic Equations on
Time Scales. Birkh\"{a}user, Boston, MA (2003)

\bibitem{Bon} Bondarenko, N.P., Buterin, S.A., Vasiliev, S.V. :An inverse
problem for Sturm -Liouville operators with frozen argument, Journal of
Mathematical Analysis and Applications, 472(1), 1028-1041, (2019)

\bibitem{But} Buterin, S., Kuznetsova, M. On the inverse problem for
Sturm--Liouville-type operators with frozen argument: rational case, Comp.
Appl. Math., 39(5), (2020)

\bibitem{Davidson} Davidson, F.A., Rynne, B.P.: Global bifurcation on time
scales. J. Math. Anal. Appl. 267, 345--360 (2002)

\bibitem{Davidson2} Davidson, F.A., Rynne, B.P.: Self-adjoint boundary value
problems on time scales. Electron. J. Differ. Equ. 175, 1--10 (2007)

\bibitem{Davidson3} Davidson, F.A., Rynne, B.P.: Eigenfunction expansions in 
$L^{2}$ spaces for boundary value problems on time-scales. J. Math. Anal.
Appl. 335, 1038--1051 (2007)

\bibitem{Erbe} Erbe, L., Hilger, S.: Sturmian theory on measure chains.
Differ. Equ. Dyn. Syst. 1, 223--244 (1993)

\bibitem{Erbe2} Erbe, L., Peterson, A.: Eigenvalue conditions and positive
solutions. J. Differ. Equ. Appl. 6, 165--191 (2000)

\bibitem{Guseinov} Guseinov, G.S.: Eigenfunction expansions for a
Sturm-Liouville problem on time scales. Int. J. Differ. Equ. 2,
93--104.(2007)

\bibitem{Guseinov2} Guseinov, G.S.: An expansion theorem for a
Sturm-Liouville operator on semi-unbounded time scales. Adv. Dyn. Syst.
Appl. 3, 147--160.(2008)

\bibitem{hu} Hu, Y.T., Bondarenko, N.P., Yang, C.F.: Traces and inverse
nodal problem for Sturm--Liouville operators with frozen argument, Applied
Mathematics Letters, 102, 106096, (2020)

\bibitem{Simon} Hilscher, R.S., Zemanek, P.: Weyl-Titchmarsh theory for time
scale symplectic systems on half line. Abstr. Appl. Anal. Art. ID 738520, 41
pp.(2011)

\bibitem{Huseinov} Huseynov, A.: Limit point and limit circle cases for
dynamic equations on time scales. Hacet. J. Math. Stat. 39, 379--392 (2010)

\bibitem{Huseinov2} Huseynov, A., Bairamov, E.: On expansions in
eigenfunctions for second order dynamic equations on time scales. Nonlinear
Dyn. Syst. Theory 9, 7--88 (2009)

\bibitem{Kong} Kong, Q.: Sturm-Liouville problems on time scales with
separated boundary conditions. Results Math. 52, 111--121 (2008)

\bibitem{Krall} Krall, A.M. The development of general differential and
general differential-boundary systems Rocky Mount. J. Math., 5, 493--542,
(1975)

\bibitem{L} Lakshmikantham,V., Sivasundaram, S., Kaymakcalan B.: Dynamic
Systems on Measure Chains, Kluwer Academic Publishers, Dordrecht, (1996)

\bibitem{Niz} Nizhink, L.P.: Inverse Eigenvalue Problems for non-local Sturm
Liouville Problems, Methods Funct. Anal. Topology, 15(1), 41-47, (2009)

\bibitem{Oz} Ozkan, A.S.: Sturm-Liouville operator with parameter-dependent
boundary conditions on time scales. Electron. J. Differential Equations,
212, 1-10, (2017)

\bibitem{Oz2} Ozkan, A.S., Adalar, I.: Half-inverse Sturm-Liouville problem
on a time scale. Inverse Probl. 36, 025015, (2020)

\bibitem{Rynne} Rynne, B.P.: L2 spaces and boundary value problems on
time-scales. J. Math. Anal. Appl. 328, 1217--1236 (2007)

\bibitem{Sun} Sun, S., Bohner, M., Chen,S.: Weyl-Titchmarsh theory for
Hamiltonian dynamic systems. Abstr. Appl. Anal. Art. ID 514760, 18 pp.(2010)

\bibitem{Went} Wentzell, A.D.:On boundary conditions for multidimensional
diffusion processes Teor. Veroyatnost. Primenen. 4 172--85 (1959) (in
Russian) \newline

Wentzell A D On boundary conditions for multidimensional diffusion processes
Theory Probab. 4 164--77 (1959) (Engl. Transl.)
\end{thebibliography}
\end{document}